\documentclass[11pt]{article}
\usepackage{amsmath}
\usepackage{etoolbox}
\usepackage{amssymb,amsfonts,amsmath,amsthm}
\usepackage{epsfig}
\usepackage{amssymb,amsfonts,amsmath,amsthm,cite,verbatim}
\usepackage{xcolor}
\usepackage{tikz}
\usepackage{graphicx}

\usepackage{hyperref}
\usepackage{enumitem}
\usepackage[normalem]{ulem}
\newtheorem{theorem}{Theorem}[section]
\newtheorem{thm}[theorem]{Theorem}
\newtheorem{prop}[theorem]{Proposition}
\newtheorem{lem}[theorem]{Lemma}
\newtheorem{rem}[theorem]{Remark}

\newtheorem{cor}[theorem]{Corollary}

\makeatletter \@addtoreset{equation}{section}

\newcommand{\qbinom}[2]{\genfrac{[}{]}{0pt}{}{#1}{#2}}

\newcommand{\Mid}{\:|\:}  
\DeclareMathOperator*{\CT}{CT}

\begin{document}

\title{A generalization of Kadell's orthogonality ex-conjecture}
\author{
  {\vspace{0.2cm}}
  {Zihao Huang$^1$, Wenlong Jiang$^{2}\footnotemark[1]$, AND Yue Zhou$^3$}\\
  {\small $^{1,2,3}$ School of Mathematics and Statistics, }\\
  {\vspace{0.2cm}}
  {\small Hunan Research Center of the Basic Discipline for Analytical Mathematics, HNP-LAMA,}\\
  {\vspace{0.2cm}}
  {\small Central South University, Changsha 410083, China}\\
  { \small $^1$zihaohuang@csu.edu.cn\ \ \ $^2$jiangwenlong@csu.edu.cn\ \ \ $^3$zhouyue@csu.edu.cn}
 }

\date{\today}

\maketitle

\footnotetext[1]{*Corresponding author}

\begin{abstract}
In 2000, Kadell gave an orthogonality conjecture for a symmetric function generalization of the Zeilberger--Bressoud $q$-Dyson constant term identity. The non-zero part of Kadell's conjecture is a constant term identity indexed by a one-row weak composition $v$. This conjecture was first proved by K\'{a}rolyi, Lascoux and Warnaar in 2015. They further formulated a closed-form expression for the above constant term when all parts of the composition $v$ are distinct. In 2021, Zhou obtained a recursion for this constant term for an arbitrary composition $v$. In this paper, by categorizing the variables into two parts, we generalize Zhou's result.
\end{abstract}



\section{Introduction}\label{sec-intr}

For a positive integer $n$, a sequence of nonnegative integers $a=(a_1, a_2,\ldots,a_n)$, and variables $x:=(x_1,x_2,\dots,x_n)$, define
\begin{equation}
D(x;a;n):=\prod_{1\leq i<j\leq n}
(x_i/x_j)_{a_i}(qx_j/x_i)_{a_j},
\end{equation}
where for a positive integer $k$, $(z)_k:=(z;q)_k=(1-z)(1-zq)\cdots(1-zq^{k-1})$ is the $q$-shifted factorial and $(z)_0:=1$. In 1975, Andrews \cite{andrews1975} conjectured that 
\begin{equation}\label{q-Dyson}
\CT_{x} D(x;a;n)=
\frac{(q)_{a_1+\cdots+a_n}}{(q)_{a_1}(q)_{a_2}\cdots(q)_{a_n}},
\end{equation}
where $\CT\limits_{x}$ denotes taking the constant term with respect to $x$. When $q\rightarrow 1^-$, the constant term identity~\eqref{q-Dyson} reduces to the Dyson constant term identity~\cite{dyson}.

Andrews' conjecture was first proved combinatorially in 1985 by Zeilberger and Bressoud \cite{zeil-bres1985}. Many other simpler proofs were given later, see~\cite{gess-xin2006,KN,cai}. Now the Andrews' ex-conjecture is usually referred to as the Zeilberger--Bressoud $q$-Dyson theorem or the $q$-Dyson constant term identity.

In the theory of Macdonald polynomials, the equal-parameter case of $D(x;a;n)$, i.e., $a_1=\cdots=a_n$, is regarded as the weight function of the constant term scalar product~\cite[section 6.9]{Mac95}. Therefore, the theory of Macdonald polynomials provides a satisfactory explanation for the equal-parameter case of the $q$-Dyson identity in terms of orthogonal polynomials. A natural and long-standing problem is how to give a similar explanation for the general $q$-Dyson identity \eqref{q-Dyson} in terms of orthogonal polynomials. An important step on this problem was achieved by Kadell~\cite{kadell} in 2000. He formulated a symmetric function generalization of the $q$-Dyson identity and gave an orthogonal conjecture. See~\eqref{kadellconj} below.

To describe the conjecture, we require some notation. For an integer sequence $v=(v_1,v_2\ldots,v_n)$, we write $|v|$ for the sum of its entries, i.e., $|v|=v_1+\cdots+v_n$. We also write $v^{+}$ for the sequence obtained from $v$ by reordering its entries in weakly decreasing order. We say $v$ is a (weak) composition, if all the entries of $v$ are nonnegative integers. A partition is a sequence $\lambda={(\lambda_1,\lambda_2,\dots)}$ of nonnegative integers such that
${\lambda_1\geq \lambda_2\geq \cdots}$ and
only finitely-many $\lambda_i$ are positive.
The length of a partition $\lambda$, denoted by $\ell(\lambda)$, is defined as the number of non-zero $\lambda_i$. Let $X=\{x_1,x_2,\dots\}$ be an alphabet of countably many variables. The $r$-th complete symmetric function $h_r(X)$ may be defined in terms of its generating function as
\begin{equation}\label{e-gfcomplete}
\sum_{r\geq 0} z^r h_r(X)=\prod_{i\geq 1}
\frac{1}{1-zx_i}.
\end{equation}
Furthermore, the complete symmetric function indexed by a composition (or partition) $v=(v_1,v_2,\dots,v_k)$ is defined as
\[
h_v:=h_{v_1}\cdots h_{v_k}.
\]
For a sequence of positive integers $a:=(a_1,a_2,\dots,a_n)$, let $X^{(a)}$ be the alphabet
\begin{align*}
    X^{(a)}:=\{& x_1,x_1 q,\ldots,x_1 q^{a_1-1},\ldots,x_{n},x_{n}q ,\ldots,x_{n}q^{a_{n}-1}\}.
\end{align*}
Define a generalized $q$-Dyson constant term
\begin{equation}\label{GDyson}
D_{v,\lambda}(a;n):=\CT_{x}
x^{-v}h_{\lambda}\big(X^{(a)}\big)
\prod_{1\leq i<j\leq n}
(x_i/x_j)_{a_i}(qx_j/x_i)_{a_j}.
\end{equation}
Here $v=(v_1,v_2,\dots,v_n)\in\mathbb{Z}^{n}$, $x^v$ denotes the monomial $x_1^{v_1}\cdots x_n^{v_n}$,
and $\lambda$ is a partition such that
$|\lambda|=|v|$. (Note that if $|\lambda| \neq |v|$ then
$D_{v,\lambda}(a;n)=0$ by the homogeneity).

For the constant term~\eqref{GDyson}, Kadell~\cite[Conjecture 4]{kadell} 
conjectured that
\begin{equation}\label{kadellconj}
D_{v,(r)}(a;n)=
\begin{cases}
\displaystyle
\cfrac{q^{\sum_{i=k+1}^n a_i}(1-q^{a_k})(q^{|a|+1})_{r-1}}{(q^{|a|-a_k+1})_r}
\prod_{i=1}^n\qbinom{a_i+\cdots+a_n}{a_i}
& \text{if $v=(0^{k-1},r,0^{n-k})$}, \\[6mm]
0 & \text{otherwise},
\end{cases}
\end{equation}
where $\qbinom{n}{k}=(q^{n-k+1})_k/(q)_k$ is a $q$-binomial coefficient for nonnegative integers $n$ and $k$. In fact, Kadell only considered $v=(r,0^{n-1})$ in his original conjecture. His conjecture was proved and extended to the above form by K\'{a}rolyi, Lascoux and Warnaar in \cite[Theorem 1.3]{KLW} using the multivariable Lagrange interpolation and the key polynomials. They further obtained a closed-form expression for $D_{v,v^{+}}(a;n)$ when $v$ is a weak composition such that $v_i\neq v_j$ for all $i\neq j$.
Subsequently, Cai~\cite{cai} gave an inductive proof of Kadell's conjecture. Finally, Zhou~\cite{Zhou-JCTA} obtained a recursion for $D_{v,v^+}(a;n)$ for arbitrary nonzero weak compositions $v$.

The main objective of this paper is to generalize Kadell's orthogonality problem by categorizing the variables into two parts. The way we categorize the variables into several parts is inspired by the $q$-Baker--Forrester conjecture \cite{Baker-Forrester2}. 
In 1998, based on the ground-state wave function of multi-component Calogero--Sutherland quantum many--body system, Baker and Forester conjectured a multi-component generalization of the $q$-Morris identity~\cite[Conjecture 2.2]{Baker-Forrester2}. It in turn gave a generalization of the $q$-Selberg integral. 
By the idea to categorize the variables in the Baker--Forrester conjecture, Jiang et al. \cite{JWZZ} considered a multi-component generalization of the $q$-Dyson constant term. Their two-part case can be equivalently written as 
\begin{equation}\label{dyson-two}
    D(a;n,n_0):=\underset{x}{\mathrm{CT}} \prod_{1\leq i<j\leq n}
\big(x_i/x_j\big)_{a_i-\chi(i\leq n_0)}
\big(qx_j/x_i\big)_{a_j-\chi(i\leq n_0)},
\end{equation}
where $a:=(a_1,a_2,\dots,a_n)$ is a sequence of positive integers and $n_0\leq n$ is a nonnegative integer. Let $X^{(a;n_0)}$ be the alphabet
\begin{align*}
    X^{(a;n_0)}:=\{& x_1,x_1 q,\ldots,x_1 q^{a_1-2},\ldots,x_{n_0},x_{n_0}q ,\ldots,x_{n_0}q^{a_{n_0}-2},\\
                & x_{n_0+1}, x_{n_0+1} q, \dots ,x_{n_0+1} q^{a_{n_0+1}-1},\ldots, x_n, x_n q,\ldots, x_n q^{a_n-1}\}.
\end{align*}
For $v\in\mathbb{Z}^n$ and a partition $\lambda$ such that $|v|=|\lambda|$, define
\begin{equation}\label{GDyson1}
D_{v,\lambda}(a;n,n_0):=\CT_{x}
x^{-v}h_{\lambda}\big(X^{(a;n_0)}\big)
\prod_{1\leq i<j\leq n}
(x_i/x_j)_{a_i-\chi(i\leq n_0)}(qx_j/x_i)_{a_j-\chi(i\leq n_0)}.
\end{equation}
It is clear that $D_{v,\lambda}(a;n,0)=D_{v,\lambda}(a;n)$. In this paper, we obtain vanishing and recursive properties for $D_{v,\lambda}(a;n,n_0)$. We provide the vanishing result first.
\begin{theorem}\label{Thm-van}
    Let $v \in \mathbb{Z}^{n}$ and $\lambda$ be a partition such that $|v|=|\lambda|$.  Then $D_{v,\lambda}(a;n,n_0)=0$, if there exists an integer $0<j<n$, such that
    \begin{equation}\label{e-thm-van}
        \sum_{k=1}^j v_{i_k}-p_I(n-n_0-j+p_I)<\sum_{k=1}^j \lambda_k 
    \end{equation}
    for every $j$-element subset $I=\{i_1,\ldots,i_j\}\subseteq \{1,\ldots,n\}$. Here $p_I:=\Mid \{ i_k : i_k\leq n_0, 1\leq k\leq j\} \Mid$.
\end{theorem}
For two integer sequences $\alpha$ and $\beta$ with only
finite many nonzero parts, define $\alpha\leq \beta$ if $\alpha_1+\cdots+\alpha_i\leq \beta_1+\cdots+\beta_i$ for all $i\geq 1$. Write $\alpha<\beta$ if $\alpha\leq \beta$ but $\alpha\neq \beta$. Now for $v \in \mathbb{Z}^{n}$ and $\lambda$ a partition such that $|v|=|\lambda|$, 
the condition $\lambda \nleq v^+$ implies that there exists an integer $0<j<n$, such that $\sum_{k=1}^j\lambda_k>\sum_{k=1}^j v_k^+$. Together with $p_I(n-n_0-j+p_I) \geq 0$ for every $I=\{i_1,\ldots,i_j\}\subseteq\{1,\ldots,n\}$, we have
\[
\sum_{k=1}^j\lambda_k > \sum_{k=1}^j v_k^+ \geq \sum_{k=1}^j v_{i_k}-p_I(n-n_0-j+p_I).
\]  
Therefore, we can apply Theorem~\ref{Thm-van} to obtain the next corollary.
\begin{cor}\label{Cor-van}
Let $v \in \mathbb{Z}^{n}$ and $\lambda$ be a partition such that $|v|=|\lambda|$. Then $D_{v,\lambda}(a;n,n_0)=0$,
if $\lambda \nleq v^+$. 
\end{cor}
Note that when $n_0=0$, Corollary~\ref{Cor-van} reduces to Cai's vanishing result in \cite[Theorem 3.8]{cai}.

Now we introduce the recursive part of our main results. Given a sequence $s=(s_1,\dots,s_n)$ and a subset $I=\{k_1,k_2,\dots,k_i\}$ of $\{1,2,\dots,n\}$, define
\[
s^{(I)}:=(s_1,s_2,\dots,\widehat{s}_{k_1},\dots,\widehat{s}_{k_2},\dots,\widehat{s}_{k_i},\dots,s_n),
\]
where $\widehat{t}$ denotes the omission of $t$. When $I=\{k\}$ for 
$k\in \{1,2,\dots,n\}$, we simply write $s^{(k)}$ for $s^{(I)}$. For a positive integer $m$, a positive integer sequence $a=(a_1,\dots,a_n)$, and two subsets $I$, $J\subseteq\{1,2,\dots,n\}$, let $a_I:=\sum_{i\in I}a_i$ and
\begin{equation}\label{def-L}
L_{m,I,J}(a):=\sum_{\substack{1\leq i\leq j\leq m\\ i\in I, j\notin J}}a_j.
\end{equation}
Denote $\widetilde{a}=\widetilde{a}(n,n_0):=(a_1-1,\ldots,a_{n_0}-1,a_{n_0+1},\ldots,a_n)$ and 
$[s]=\{1,2,\dots,s\}$ for a positive integer $s$. 
Let $\qbinom{n}{t}=\frac{(q)_n}{(q)_{t_1}(q)_{t_2}\cdots(q)_{t_k}}$ be the 
$q$-multinomial coefficient for a nonnegative integer sequence 
$t=(t_1,\dots,t_k)$.
\begin{thm}\label{p-recu}
Let $v=(v_1,\dots,v_n)\in \mathbb{Z}^n$ and $\lambda$ be a partition such that $|\lambda|=|v|$. 
Define $r:=\max(\{v_i-n+n_0:i\leq n_0\}\cup \{v_i:i\geq n_0+1\})$, $S_1:=\{i:1\leq i\leq n_0
\text{ and } v_i-n+n_0=r\}$, $S_2:=\{i : n_0+1 \leq i\leq n \text{ and } v_i=r\}$, $s_1:=|S_1|$ and $s_2:=|S_2|$. We have:
\begin{enumerate}[label=(\arabic*)]
\item If $\max\{v_i-n+n_0:i\leq n_0\} \geq \max\{v_i:i\geq n_0+1\}+s_1$ and $\lambda_1=\lambda_2=\dots=\lambda_{s_1}=r$, then
\begin{multline}\label{e-recu-A}
 D_{v,\lambda}(a;n,n_0)=(-1)^{s_1(n-n_0)}\frac{\qbinom{|\widetilde{a}|+r-1}{\widetilde{a},r-1}}
 {\qbinom{|\widetilde{a}|-\widetilde{a}_{S_1}+r-1}{\widetilde{a}^{(S_1)},r-1}}
 \sum_{\emptyset \neq J\subseteq S_1}(-1)^{|S_1\setminus J|}q^{L_{n_0,S_1,J}(\widetilde{a})}
 \frac{1-q^{\widetilde{a}_{J}}}{1-q^{|\widetilde{a}|-\widetilde{a}_{J}+r}}\\\times D_{(v+s_1\alpha)^{(S_1)},\lambda^{([s_1])}}\big(a^{(S_1)};n-s_1,n_0-s_1\big),
\end{multline}
where $\alpha=(\underset{n_0}{\underbrace{0,\dots,0}},\underset{n-n_0}{\underbrace{1,\dots,1}})$. 

\item If $\max\{v_i:i\geq n_0+1\} \geq \max\{v_i-n+n_0:i\leq n_0\}+ s_2$ and $\lambda_1=\lambda_2=\dots=\lambda_{s_2}=r$, then
      \begin{multline}\label{e-recu-B}
D_{v,\lambda}(a;n,n_0)=\frac{\qbinom{|\widetilde{a}|+r-1}{\widetilde{a},r-1}}
{\qbinom{|\widetilde{a}|-\widetilde{a}_{S_2}+r-1}{\widetilde{a}^{(S_2)},r-1}}\sum_{\emptyset \neq J\subseteq S_2}(-1)^{|S_2\setminus J|}q^{L_{n,S_2,J}(\widetilde{a})}
\frac{1-q^{\widetilde{a}_{J}}}{1-q^{|\widetilde{a}|-\widetilde{a}_{J}+r}}
\\\times D_{v^{(S_2)},\lambda^{([s_2])}}\big(a^{(S_2)};n-s_2,n_0\big).
\end{multline}
\end{enumerate}
 
\end{thm}


Our original motivation is to investigate the case $\lambda=v^+$.
This can be easily obtained by Theorem \ref{Thm-van} and Theorem \ref{p-recu}.
\begin{cor}\label{Thm-rec}
With the same notation as above, we have: 
\begin{enumerate}[label=(\arabic*)]
\item  If $n_0 < n$ and $\min\{v_i:n_0+1\leq i \leq n\} < \max\{v_i:1\leq i \leq n_0\}$, then
$D_{v,v^+}(a;n,n_0)=0;$   
\item If $n_0 < n$ and $\min\{v_i:n_0+1\leq i \leq n\} \geq \max\{v_i:1\leq i \leq n_0\}$, then
\begin{multline}\label{e-thm-recu1}
D_{v,v^+}(a;n,n_0)=\frac{\qbinom{|\widetilde{a}|+r-1}{\widetilde{a},r-1}}
{\qbinom{|\widetilde{a}|-\widetilde{a}_{S_2}+r-1}{\widetilde{a}^{(S_2)},r-1}}
\sum_{\emptyset \neq J\subseteq S_2}(-1)^{|S_2\setminus J|}q^{L_{n,S_2,J}(\widetilde{a})}
\frac{1-q^{\widetilde{a}_{J}}}{1-q^{|\widetilde{a}|-\widetilde{a}_{J}+r}}\\
\times D_{v^{(S_2)},(v^{(S_2)})^+}\big(a^{(S_2)};n-s_2,n_0\big);
\end{multline}
\item  If $n_0=n$, then
\begin{multline}\label{e-thm-recu2}
    D_{v,v^+}(a;n,n)=\frac{\qbinom{|\widetilde{a}|+r-1}{\widetilde{a},r-1}}
{\qbinom{|\widetilde{a}|-\widetilde{a}_{S_1}+r-1}{\widetilde{a}^{(S_1)},r-1}}
\sum_{\emptyset \neq J\subseteq S_1}(-1)^{|S_1\setminus J|}q^{L_{n,S_1,J}(\widetilde{a})}
\frac{1-q^{\widetilde{a}_{J}}}{1-q^{|\widetilde{a}|-\widetilde{a}_{J}+r}}\\
\times D_{v^{(S_1)},(v^{(S_1)})^+}\big(a^{(S_1)};n-s_1,n-s_1\big). 
\end{multline}
\end{enumerate}
\end{cor}
Note that i) in Corollary~\ref{Thm-rec}, the case (2) when $n_0=0$ is essentially the case (3), and the case (3) is just Zhou's recursion for $D_{v,v^+}(a;n)$ \cite[Theorem 1.1]{Zhou-JCTA};
ii) the complete symmetric functions $h_{\lambda}$ in Theorem~\ref{Thm-van} and  Corollary~\ref{Thm-rec} can be replaced by the Schur functions $s_{\lambda}$.
We will explain this in Section~\ref{Schur}.

The method employed to prove Theorem~\ref{Thm-van} and Theorem~\ref{p-recu}
is based on Cai's splitting idea \cite{cai}. Using the idea we obtain an explicit splitting formula (see \eqref{e-Fsplit} below) for the rational function,
\begin{equation}\label{d-F}
F_{n,n_0}(a;x,w):=\prod_{1\leq i<j\leq n}
(x_i/x_j)_{a_i-\chi(i\leq n_0)}(qx_j/x_i)_{a_j-\chi(i\leq n_0)}
\prod_{i=1}^n\prod_{j=1}^s(x_i/w_j)_{a_i-\chi(i\leq n_0)}^{-1},
\end{equation}
where $w:=(w_1,\dots,w_s)$ is a sequence of parameters.
Throughout this paper, we assume that the $w_j$ are large enough so that all terms of the form $cx_i/w_j$ in \eqref{d-F}
satisfy $|cx_i/w_j|<1$, where $c\in \mathbb{C}(q)\setminus \{0\}$. Hence, we can write
\[
\frac{1}{1-cx_i/w_j}=\sum_{k\geq 0} (cx_i/w_j)^k.
\]
By the generating function of complete symmetric functions \eqref{e-gfcomplete}, it is easy to see that the constant term $D_{v,\lambda}(a;n,n_0)$ equals a certain coefficient of $F_{n,n_0}(a;x,w)$. That is
\begin{equation}\label{e-relation}
D_{v,\lambda}(a;n,n_0)=\CT_{x,w}x^{-v}w^{\lambda}F_{n,n_0}(a;x,w),
\end{equation}
where the length of $w$ should be no less than the length of $\lambda$.
In \cite{cai}, Cai gave a splitting formula for $F_{n,0}(a;x,w)$. In this paper, we will give a splitting formula for $F_{n,n_0}(a;x,w)$, which is equivalent to the two-part case splitting formula in~\cite{JWZZ}. Using the splitting formula, we prove Theorem~\ref{Thm-van} and obtain an inductive formula for certain $D_{v,\lambda}(a;n,n_0)$, see Lemma ~\ref{lem-ind} below. 

The remainder of this paper is organized as follows.
In the next section,
we obtain the splitting formula for $F_{n,n_0}(a;x,w)$.
In Section~\ref{sec-proof-van}, we prove Theorem~\ref{Thm-van} by induction.
In Section~\ref{sec-inductive}, we obtain an inductive formula for certain $D_{v,\lambda}(a;n,n_0)$ using the splitting formula.
In Section~\ref{sec-proof}, we prove Theorem~\ref{p-recu}. 
In Section~\ref{Schur}, we give other forms of Theorem~\ref{Thm-van} and  Corollary~\ref{Thm-rec}.

\section{A splitting formula for $F_{n,n_0}(a;x,w)$}\label{sec-splitting}

In this section, we give an explicit splitting formula for the rational function $F_{n,n_0}(a;x,w)$. It reduces to Cai's splitting formula \cite[Theorem 3.3]{cai} when $n_0=0$.

The proof of the splitting formula relies on the following straightforward results.
\begin{lem}
Let $i,j$ be nonnegative integers and $t$ be an integer. Then
\begin{subequations}\label{e-ab}
\begin{equation}\label{prop-b1}
\frac{(1/y)_{i}(qy)_j}{(q^{-t}/y)_i}=q^{it}(q^{1-i}y)_t(q^{t+1}y)_{j-t} \quad \text{ for $0\leq t\leq j$},
\end{equation}
\begin{equation}\label{prop-b2}
\frac{(y)_j(q/y)_i}{(q^{-t}/y)_i}=q^{i(t+1)}(q^{-i}y)_{t+1}(q^{t+1}y)_{j-t-1} \quad \text{ for $-1\leq t\leq j-1$},
\end{equation}
\begin{equation}\label{prop-c}
\frac{(y)_j(q/y)_i}{(q^{-t}/y)_{i+1}}=-yq^{(i+1)t}(q^{-i}y)_t(q^{t+1}y)_{j-t-1} \quad \text{ for $0\leq t\leq j-1$}.
\end{equation}
\end{subequations}
\end{lem}
Note that the case $t=j$ of \eqref{prop-b1} (taking $y\mapsto y/q$) is the standard fact in \cite[Equation~(I.13)]{GR}.
The condition $0\leq t\leq j-1$ in \eqref{prop-c} forces $j$ to be a positive integer.
\begin{proof}
For $0\leq t\leq j$,
\[
\frac{(1/y)_{i}(qy)_j}{(q^{-t}/y)_i}=\frac{(q^{i-t}/y)_t(qy)_j}{(q^{-t}/y)_t}
=\frac{(-1/y)^tq^{it-\binom{t+1}{2}}(q^{1-i}y)_t(qy)_j}{(-1/y)^{t}q^{-\binom{t+1}{2}}(qy)_t}
=q^{it}(q^{1-i}y)_t(q^{t+1}y)_{j-t}.
\]

Taking $y\mapsto y/q$ and $t\mapsto t+1$ in \eqref{prop-b1} yields
\eqref{prop-b2} for $-1\leq t\leq j-1$. For $0\leq t\leq j-1$,
\[
\frac{(y)_j(q/y)_i}{(q^{-t}/y)_{i+1}}=
\frac{(y)_j(q^{i-t+1}/y)_t}{(q^{-t}/y)_{t+1}}=
\frac{(y)_j(-1/y)^tq^{it-\binom{t}{2}}(q^{-i}y)_t}{(-1/y)^{t+1}q^{-\binom{t+1}{2}}(y)_{t+1}}=
-yq^{(i+1)t}(q^{-i}y)_t(q^{t+1}y)_{j-t-1}. \qedhere
\]
\end{proof}

By partial fraction decomposition, we have the following splitting formula for the rational function $F_{n,n_0}(a;x,w)$.
\begin{prop}
Let $F_{n,n_0}(a;x,w)$ be defined as \eqref{d-F}. Then
\begin{equation}\label{e-Fsplit}
F_{n,n_0}(a;x,w)=\sum_{i=1}^{n_0}\sum_{j=0}^{a_i-2}\frac{A_{ij}}{1-q^jx_i/w_1}+\sum_{i=n_0+1}^{n}\sum_{j=0}^{a_i-1}\frac{B_{ij}}{1-q^jx_i/w_1},
\end{equation}
where
\begin{align}\label{A}
A_{ij}&=
\frac{q^{j(|a|-a_i-n_0+1)+\sum_{l=i+1}^{n_0}(a_l-1)}}{(q^{-j})_{j}(q)_{a_i-j-2}\prod_{t=2}^s{\big(x_i/w_t\big)_{a_i-1}}}F_{n-1,n_0-1}(a^{(i)};x^{(i)},w^{(1)})
\\\nonumber&\quad\times \prod_{l=1}^{i-1}\big(q^{2-a_l}x_i/x_l\big)_{j}\big(q^{j+1} x_i/x_l\big)_{a_i-j-1}\prod_{l=i+1}^{n_0}\big(q^{1-a_l}x_i/x_l\big)_{j+1}\big(q^{j+1}x_i/x_l\big)_{a_i-j-2}
\\\nonumber&\quad\times\prod_{l=n_0+1}^{n}(-x_i/x_l)\big(q^{1-a_l}x_i/x_l\big)_{j}\big(q^{j+1}x_i/x_l\big)_{a_i-j-2},
\end{align}
and 
\begin{align}\label{B}
B_{ij}&=\frac{q^{j(|a|-a_i-n_0)+\sum_{l=i+1}^{n}a_l}}{(q^{-j})_{j}(q)_{a_i-j-1}\prod_{t=2}^{s}(x_i/w_t)_{a_i}}F_{n-1,n_0}(a^{(i)};x^{(i)},w^{(1)})
\\\nonumber&\quad\times \prod_{l=1}^{n_0}\big(q^{2-a_l}x_i/x_l\big)_{j}\big(q^{j+1} x_i/x_l\big)_{a_i-j-1}\prod_{l=n_0+1}^{i-1}\big(q^{1-a_l}x_i/x_l\big)_{j}\big(q^{j+1}x_i/x_l\big)_{a_i-j}\
\\\nonumber&\quad\times\prod_{l=i+1}^{n}\big(q^{-a_l}x_i/x_l\big)_{j+1}\big(q^{j+1}x_i/x_l\big)_{a_i-j-1}.\
\end{align}
\end{prop}
Note that $A_{ij}$ is a power series in $x_i$ whose lowest degree is at least $n-n_0$, and $B_{ij}$ is a power series in $x_i$ whose lowest degree is at least $0$.

\begin{proof}
By the partial fraction decomposition, we can rewrite $F_{n,n_0}(a;x,w)$ as \eqref{e-Fsplit}, where
\begin{align}\label{Aij}
    A_{ij}=F_{n,n_0}(a;x,w)(1-q^j x_i/w_1)\Mid_{w_1=q^j x_i} \text{ for } i=1,\ldots,n_0 \text{ and }j=0,\ldots,a_i-2,
\end{align}
and
\begin{align}\label{Bij}
    B_{ij}=F_{n,n_0}(a;x,w)(1-q^j x_i/w_1)\Mid_{w_1=q^j x_i} \text{ for } i=n_0+1,\ldots,n \text{ and }j=0,\ldots,a_i-1.
\end{align}
The substitution in~\eqref{Aij} yields
\begin{multline}\label{A1}
   A_{ij}
   =\frac{F_{n-1,n_0-1}(a^{(i)};x^{(i)},w^{(1)})}{(q^{-j})_{j}(q)_{a_i-j-2}\prod_{t=2}^s{\big(x_i/w_t\big)_{a_i-1}}} \prod_{l=1}^{i-1}\cfrac{(x_l/x_i)_{a_l-1}(qx_i/x_l)_{a_i-1}}{(q^{-j}x_l/x_i)_{a_l-1}}
   \\  \times\prod_{l=i+1}^{n_0}\cfrac{(x_i/x_l)_{a_i-1}(qx_l/x_i)_{a_l-1}}{(q^{-j}x_l/x_i)_{a_l-1}}\prod_{l=n_0+1}^{n}\cfrac{(x_i/x_l)_{a_i-1}(qx_l/x_i)_{a_l-1}}{(q^{-j}x_l/x_i)_{a_l}}.
\end{multline}
Taking $(i,j,t,y)\mapsto (a_l-1,a_i-1,j,x_i/x_l)$ in \eqref{e-ab}, we have: for $j=0,\dots,a_i-1$,
\begin{subequations}\label{e-B}
\begin{equation}\label{B1}
\frac{(x_l/x_i)_{a_l-1}(qx_i/x_l)_{a_i-1}}{(q^{-j}x_l/x_i)_{a_l-1}}
=q^{j(a_l-1)}\big(q^{2-a_l}x_i/x_l\big)_j\big(q^{j+1}x_i/x_l\big)_{a_i-j-1},
\end{equation}
for $j=-1,\dots,a_i-2$,
\begin{equation}\label{B2}
\frac{(x_i/x_l)_{a_i-1}(qx_l/x_i)_{a_l-1}}{(q^{-j}x_l/x_i)_{a_l-1}}
=q^{(j+1)(a_l-1)}\big(q^{1-a_l}x_i/x_l\big)_{j+1}\big(q^{j+1}x_i/x_l\big)_{a_i-j-2},
\end{equation}
and for $j=0,\dots,a_i-2$
\begin{equation}
\frac{(x_i/x_l)_{a_i-1}(qx_l/x_i)_{a_l-1}}{(q^{-j}x_l/x_i)_{a_l}}=
-q^{j a_l}x_i/x_l(q^{1-a_l}x_i/x_l)_j(q^{j+1}x_i/x_l)_{a_i-j-2} .
\end{equation}
\end{subequations}
 Substituting~\eqref{e-B} into~\eqref{A1} we obtain~\eqref{A}. 

Similarly, carrying out the substitution $w=q^jx_i$ in $F_{n,n_0}(a;x,w)(1-q^jx_i/w)$ for $i=n_0+1,\dots,n$ yields
\begin{multline}\label{BB}
B_{ij}=\frac{F_{n-1,n_0}(a^{(i)};x^{(i)},w^{(1)})}{(q^{-j})_j(q)_{a_i-j-1}}\prod_{l=1}^{n_0}\frac{(x_l/x_i)_{a_l-1}(qx_i/x_l)_{a_i-1}}{(q^{-j}x_l/x_i)_{a_l-1}}\\
\quad\times 
\prod_{l=n_0+1}^{i-1}\frac{(x_l/x_i)_{a_l}(qx_i/x_l)_{a_i}}{(q^{-j}x_l/x_i)_{a_l}}
\prod_{l=i+1}^n\frac{(x_i/x_l)_{a_i}(qx_l/x_i)_{a_l}}{(q^{-j}x_l/x_i)_{a_l}}.
\end{multline}
Taking $a_l \mapsto a_l+1$ and $a_i \mapsto a_i+1$ in \eqref{B1} and \eqref{B2}, respectively, we have
\begin{equation}\label{B3}
\frac{(x_l/x_i)_{a_l}(qx_i/x_l)_{a_i}}{(q^{-j}x_l/x_i)_{a_l}}
=q^{j a_l}\big(q^{1-a_l}x_i/x_l\big)_j\big(q^{j+1}x_i/x_l\big)_{a_i-j}
\quad \text{for $j=0,\dots,a_i$},
\end{equation}
and
\begin{equation}\label{B4}
\frac{(x_i/x_l)_{a_i}(qx_l/x_i)_{a_l}}{(q^{-j}x_l/x_i)_{a_l}}
=q^{(j+1)a_l}\big(q^{-a_l}x_i/x_l\big)_{j+1}\big(q^{j+1}x_i/x_l\big)_{a_i-j-1}
\quad \text{for $j=-1,\dots,a_i-1$}.
\end{equation}
Substituting \eqref{B1}, \eqref{B3} and \eqref{B4} into \eqref{BB} we obtain \eqref{B}.
\end{proof}

\section{Proof of Theorem~\ref{Thm-van}}\label{sec-proof-van}
In this section, we give a proof of Theorem~\ref{Thm-van}. By~\eqref{e-relation}, it suffices to prove that
\[
\underset{x,w}{\CT} x^{-v} w^\lambda F_{n,n_0}(a;x,w)=0,
\]
if there exists an integer $0<j<n$ such that \begin{equation}\label{e-thm-van-2}
        \sum_{k=1}^j v_{i_k}-p_I(n-n_0-j+p_I)<\sum_{k=1}^j \lambda_k
\end{equation}
for every $I=\{i_1,\ldots,i_j\} \subseteq \{1,\ldots,n\}$. Here we recall that $p_I:=\Mid \{ i_k : i_k\leq n_0, 1\leq k\leq j\} \Mid$ for a given $I$.

\begin{proof}[Proof of Theorem~\ref{Thm-van}]
The proof proceeds by induction on $n$. The base case is $n=2$. In this case, $0<j<2$ forces $j=1$, and that \eqref{e-thm-van-2} holds for each $I=\{i\}\subseteq \{1,2\}$ implies that
\[
v_{i}-\chi(i\leq n_0)(2-n_0)<\lambda_1\quad \text{ for $i=1,2$}.
\]
By the splitting formula~\eqref{e-Fsplit} for $F_{2,n_0}(a;x,w)$, we have
\begin{equation}
    x^{-v}w^{\lambda} F_{2,n_0}(a;x,w)=\sum_{i=1}^{n_0}\sum_{j=0}^{a_i-2}\frac{x^{-v}w^{\lambda} A_{ij} }{1-q^jx_i/w_1}+\sum_{i=n_0+1}^{2}\sum_{j=0}^{a_i-1}\frac{x^{-v}w^{\lambda} B_{ij}}{1-q^jx_i/w_1}.
\end{equation}
Taking the constant term with respect to $w_1$ in both sides of the above equation, we obtain 
\[
\underset{w_1}{\CT}x^{-v}w^{\lambda} F_{2,n_0}(a;x,w)=\sum_{i=1}^{n_0} \sum_{j=0}^{a_i-2} (q^{j}x_i)^{\lambda_1} x^{-v} w^{\lambda^{(1)}} A_{ij} + \sum_{i=n_0+1}^{2} \sum_{j=0}^{a_i-1} (q^{j} x_i)^{\lambda_1} x^{-v} w^{\lambda^{(1)}} B_{ij},
\]
where $w^{\lambda^{(1)}}=w_2^{\lambda_2}\cdots w_s^{\lambda_s}$.
Since $A_{ij}$ and $B_{ij}$ are both power series in $x_i$ with lowest  degrees at least $2-n_0$ and $0$, respectively, each summand above vanishes when taking the constant term with respect to $x_i$. Hence $\underset{x,w}{\CT}x^{-v}w^{\lambda} F_{2,n_0}(a;x,w)=0$.

Now assuming the theorem holds for $n-1$, we prove it holds for $n$. Using the splitting formula~\eqref{e-Fsplit} for $F_{n,n_0}(a;x,w)$ and taking the constant term with respect to $w_1$, we obtain
\begin{multline}\label{eq-van-eq1}
    \underset{x,w}{\CT}x^{-v}w^{\lambda} F_{n,n_0}(a;x,w)=\sum_{i=1}^{n_0} \sum_{j=0}^{a_i-2} \underset{x,w}{\CT} (q^{j}x_i)^{\lambda_1} x^{-v}  w^{\lambda^{(1)}} A_{ij} \\ + \sum_{i=n_0+1}^{n} \sum_{j=0}^{a_i-1} \underset{x,w}{\CT} (q^{j} x_i)^{\lambda_1} x^{-v} w^{\lambda^{(1)}} B_{ij}.
\end{multline}
To simplify notation, let
\[
     C_{ij}:=
    \begin{cases}
       A_{ij} & \text{when $i\leq n_0$},\\
       B_{ij} & \text{when $i>n_0$}.
    \end{cases}
\]
Then $C_{ij}$ is a power series in $x_i$. For $i\leq n_0$, its lowest degree in $x_i$ is at least $n-n_0$. For $i>n_0$, its lowest degree in $x_i$ is at least $0$.  
    
For $j=1$, if~\eqref{e-thm-van-2} holds for every $1$-element subset $I=\{i\}\subseteq \{1,\ldots,n\}$, then
    \[
    v_i-\chi(i\leq n_0)(n-n_0)<\lambda_1\quad \text{ for $i=1,\ldots,n$}.
    \]
This gives that each summand $\underset{x,w}{\CT}(q^{j}x_i)^{\lambda_1} x^{-v}  w^{\lambda^{(1)}} C_{ij}$ in \eqref{eq-van-eq1} vanishes by a degree consideration in $x_i$. Hence $\underset{x,w}{\CT}x^{-v}w^\lambda F_{n,n_0}(a;x,w)=0$. Given an integer $j\geq 2$, assume that \eqref{e-thm-van-2} holds for all $j$-elements subsets $I=\{i_1,\ldots,i_j\}\subseteq\{1,\ldots,n\}$. We prove $\underset{x,w}{\CT}x^{-v}w^{\lambda} F_{n,n_0}(a;x,w)=0$ by showing that each summand $\underset{x,w}{\CT} (q^{j}x_i)^{\lambda_1} x^{-v}  w^{\lambda^{(1)}} C_{ij}$ in \eqref{eq-van-eq1} vanishes. It is clear that
\[
  \underset{x_i}{\CT}(q^{j}x_i)^{\lambda_1} x^{-v}  w^{\lambda^{(1)}} C_{ij}= 0
\]
for $v_i < \lambda_1$. Thus, in the subsequent proof we assume $v_i\geq \lambda_1$. By expanding the denominator $\prod_{t=2}^s (x_i/w_t)_{a_i-\chi(i\leq n_0)}$ of $C_{ij}$ as a power series, and taking the constant term with respect to $x_i$, we can write $\underset{x_i}{\CT}(q^{j}x_i)^{\lambda_1} x^{-v}  w^{\lambda^{(1)}} C_{ij}$ as a finite sum of the form
\begin{equation}\label{van-thm-induction-case1}
    c_q \cdot x_1^{-v_1^\prime} \cdots \hat{x_i} \cdots x_n^{-v_n^\prime} \cdot w_2^{\lambda_2^\prime} \cdots w_s^{\lambda_s^\prime} \cdot F_{n-1, n_0 - \chi(i \leq n_0)} (a^{(i)};x^{(i)}, w^{(1)}).
\end{equation}
Here $c_q$ is a constant, $v^\prime = (v_1^\prime, \ldots, \hat{v_i}, \ldots, v_n^\prime)$ and $\lambda^\prime = (\lambda_2^\prime, \ldots, \lambda_s^\prime)$ are integer sequences satisfying
\begin{subequations}\label{v-lambda-prime}
    \begin{equation}\label{v-prime-relation}
         v_k^\prime \geq v_k+\chi(i\leq n_0)\cdot \chi(k > n_0)\quad \text{ for $k\neq i$},
    \end{equation}
    \begin{equation}\label{lambda-prime-relation}
        \lambda_l^{\prime} \leq \lambda_l \quad \text{ for $l=2,\ldots,s$},
    \end{equation}
    and
    \begin{equation}\label{v-lambda-sum}
         \sum_{\substack{k=1\\ k\neq i}}^n v_k^\prime - \sum_{k=2}^s \lambda_k^\prime = \sum_{\substack{k=1\\ k\neq i}}^n v_k - \sum_{k=2}^s \lambda_k+\lambda_1-v_i.
    \end{equation}
\end{subequations}
In the following, we prove $\underset{x,w}{\CT}(q^{j}x_i)^{\lambda_1} x^{-v}  w^{\lambda^{(1)}} C_{ij}=0$ by showing that each summand 
    \begin{equation}\label{eq-proof-thm1-induc-F}
    \CT_{x,w} c_q \cdot x_1^{-v_1^\prime} \cdots \hat{x_i} \cdots x_n^{-v_n^\prime} \cdot w_2^{\lambda_2^\prime} \cdots w_s^{\lambda_s^\prime} \cdot F_{n-1, n_0 - \chi(i \leq n_0)} (a^{(i)};x^{(i)}, w^{(1)})=0.
    \end{equation}
The cases when \(\lambda^\prime\) has negative entries is trivial, so we may assume that \(\lambda^\prime\) is a composition. Furthermore, since the constant term is invariant under permutations of $w$, we may assume without loss of generality that \(\lambda^\prime\) is a partition.

For any fixed $j-1$ elements set $I_1=\{i_1,\ldots,i_{j-1}\} \subset \{1,\ldots,n\}\setminus\{i\}$, by ~\eqref{v-lambda-prime}, we have
    \begin{equation}\label{v-lambda-j-1-subset-inequality}
        \sum_{k=1}^{j-1} (v_k^\prime - v_k)+\sum_{k=2}^{j}(\lambda_k-\lambda_k^\prime)\leq \lambda_1-v_i-\chi(i\leq n_0)(n-n_0-j+1+p_{I_1}).
    \end{equation}
    By \eqref{e-thm-van-2}, ~\eqref{v-lambda-j-1-subset-inequality} and the assumption $v_i\geq \lambda_1$, we have
    \begin{align}\label{e-thm-van-induction}
    &\sum_{k=1}^{j-1}(v_{i_k}^\prime - \lambda_{k+1}^\prime)\\
    & \leq \sum_{k=1}^{j-1} (v_{i_k}-\lambda_{k+1})+\lambda_1-v_i-\chi(i\leq n_0)(n-n_0-j+1+p_{I_1})\nonumber\\
    & \leq \sum_{k=1}^j (v_{i_k}-\lambda_k)-\chi(i\leq n_0)(n-n_0-j+1+p_{I_1})\nonumber\\
    & < (p_{I_1}+\chi(i\leq n_0))(n-n_0-j+p_{I_1}+\chi(i\leq n_0))-\chi(i\leq n_0)(n-n_0-j+1+p_{I_1})\nonumber\\
    & \leq p_{I_1}(n-n_0+\chi(i\leq n_0)-j+p_{I_1}+1).\nonumber
    \end{align}
    Now \eqref{e-thm-van-induction} implies that~\eqref{eq-proof-thm1-induc-F} satisfies the vanishing condition with  
    $n\mapsto n-1$, $n_0\mapsto n_0-\chi(i\leq n_0)$, $(x_1,\ldots,x_{n-1}) \mapsto x^{(i)}$, $v\mapsto v^\prime$, $w\mapsto w^{(1)}$, $\lambda\mapsto\lambda^\prime$, and $j\mapsto j-1$. By the inductive assumption we get~\eqref{eq-proof-thm1-induc-F}. Thereby all the summands $\underset{x,w}{\CT}(q^{j}x_i)^{\lambda_1} x^{-v}  w^{\lambda^{(1)}} C_{ij}$ in ~\eqref{eq-van-eq1} vanish, which implies that $\underset{x,w}{\CT} x^{-v}w^\lambda F_{n,n_0}(a;x,w)=0$. This shows that the theorem holds for $n$. By the induction, the theorem holds.
\end{proof}

\section{An inductive formula for $D_{v,\lambda}(a;n,n_0)$}\label{sec-inductive}

In this section, we use the splitting formula \eqref{e-Fsplit} for $F_{n,n_0}(a;x,w)$
to give an inductive formula for certain $D_{v,\lambda}(a;n,n_0)$.

We first need a simple result deduced from the $q$-binomial theorem.
\begin{prop}\label{prop-sum}
Let $n$ and $t$ be nonnegative integers. Then
\begin{equation}\label{e-sum}
\sum_{k=0}^{t}
\frac{q^{k(n-t)}}
{(q^{-k})_{k}(q)_{t-k}}
=\qbinom{n}{t}.
\end{equation}
\end{prop}
\begin{proof}
We can rewrite the left-hand side of \eqref{e-sum} as
\begin{equation}\label{e-sum2}
\sum_{k=0}^{t}
\frac{(-1)^kq^{k(n-t)+\binom{k+1}{2}}}{(q)_{k}(q)_{t-k}}
=\frac{1}{(q)_t} \sum_{k=0}^{t}q^{\binom{k}{2}}\qbinom{t}{k}(-q^{n-t+1})^k.
\end{equation}
Using the well-known $q$-binomial theorem \cite[Theorem 3.3]{andrew-qbinomial}, we obtain
\[
(z)_t=\sum_{k=0}^tq^{\binom{k}{2}}\qbinom{t}{k}(-z)^k.
\]
Substituting $z\mapsto q^{n-t+1}$, we have
\[
\frac{1}{(q)_t} \sum_{k=0}^{t}q^{\binom{k}{2}}\qbinom{t}{k}(-q^{n-t+1})^k
=(q^{n-t+1})_t/(q)_t=\qbinom{n}{t}. \qedhere
\]
\end{proof}

By the splitting formula~\eqref{e-Fsplit} for $F_{n,n_0}(a;x,w)$, we can obtain the following inductive formula for $D_{v,\lambda}(a;n,n_0)$.
\begin{lem}\label{lem-ind}
Let $v=(v_1,\dots,v_n)\in \mathbb{Z}^n$ and $\lambda$ be a partition such that $|\lambda|=|v|$. Denote $r:=\max(\{v_i-n+n_0:i\leq n_0\}\cup\{v_i:i\geq n_0+1\})$, $S_1:=\{i\in\{1,\ldots,n_0\}:v_i-n+n_0=r\}$ and $S_2:=\big\{i\in \{n_0+1,\dots,n\}: v_i=r\big\}$. If $\lambda_1= r$, then
\begin{multline}\label{e-ind}
D_{v,\lambda}(a;n,n_0)=\sum_{i\in S_1} (-1)^{n-n_0} q^{\sum_{l=i+1}^{n_0} (a_l-1)}\qbinom{|a|+r-1-n_0}{a_{i}-2}D_{(v+\alpha)^{(i)},\lambda^{(1)}}\big(a^{(i)};n-1,n_0-1\big)\\
+ \sum_{i\in S_2} q^{\sum_{l=i+1}^{n} a_l}\qbinom{|a|+r-1-n_0}{a_{i}-1}D_{v^{(i)},\lambda^{(1)}}\big(a^{(i)};n-1,n_0\big),
\end{multline}
where $\alpha=(\underset{n_0}{\underbrace{0,\dots,0}},\underset{n-n_0}{\underbrace{1,\dots,1}}).$
\end{lem}
\begin{proof}
Substituting the splitting formula \eqref{e-Fsplit} for $F_{n,n_0}(a;x,w)$ into~\eqref{e-relation}, 
we have
\[
D_{v,\lambda}(a;n,n_0)=\sum_{i=1}^{n_0}\sum_{j=0}^{a_i-2}\CT_{x,w}
\frac{x^{-v}w^{\lambda}A_{ij}}{1-q^jx_i/w_1}
+\sum_{i=n_0+1}^{n}\sum_{j=0}^{a_i-1}\CT_{x,w}
\frac{x^{-v}w^{\lambda}B_{ij}}{1-q^jx_i/w_1},
\]
where $w=(w_1,\dots,w_n)$.
Taking the constant term with respect to $w_1$ yields
\begin{align}\label{e-ind1}
D_{v,\lambda}(a;n,n_0)
 &=\sum_{i=1}^{n_0}\sum_{j=0}^{a_i-2}\CT_{x,w^{(1)}}q^{j\lambda_1}x^{-v}x_i^{\lambda_1}w^{\lambda^{(1)}}A_{ij}
 +\sum_{i=n_0+1}^{n}\sum_{j=0}^{a_i-1}\CT_{x,w^{(1)}}q^{j\lambda_1}x^{-v}x_i^{\lambda_1}w^{\lambda^{(1)}}B_{ij}\\
 &\nonumber=\sum_{i=1}^{n_0}\sum_{j=0}^{a_i-2}\CT_{x,w^{(1)}}q^{jr}x^{-v}x_i^{r}w^{\lambda^{(1)}}A_{ij}
 +\sum_{i=n_0+1}^{n}\sum_{j=0}^{a_i-1}\CT_{x,w^{(1)}}q^{jr}x^{-v}x_i^{r}w^{\lambda^{(1)}}B_{ij},
\end{align}
where $w^{\lambda^{(1)}}=w_2^{\lambda_2}
\cdots w_n^{\lambda_n}$. Recall that $A_{ij}$ and $B_{ij}$ are power series in $x_i$ with lowest degrees at least $n-n_0$ and $0$, respectively. It follows that:
\[
    \CT_{x,w^{(1)}}q^{jr}x^{-v}x_i^rw^{\lambda^{(1)}}A_{ij}=0 \quad \text{for $1\leq i\leq n_0$ and $i \notin S_1$,}
\]
\[
\CT_{x,w^{(1)}}q^{jr}x^{-v}x_i^rw^{\lambda^{(1)}}B_{ij}=0 \quad \text{for $n_0< i\leq n$ and $i \notin S_2$,}
\]
\[
    \CT_{x,w^{(1)}}q^{jr}x^{-v}x_i^rw^{\lambda^{(1)}}A_{ij}=\CT_{x^{(i)},w^{(1)}}q^{jr}x^{-v^{(i)}}w^{\lambda^{(1)}}(x_i^{-n+n_0} A_{ij})|_{x_i=0}
\quad \text{for $i \in S_1$,}
\]
and
\[\CT_{x,w^{(1)}}q^{jr}x^{-v}x_i^rw^{\lambda^{(1)}}B_{ij}=\CT_{x^{(i)},w^{(1)}}q^{jr}x^{-v^{(i)}}w^{\lambda^{(1)}}B_{ij}|_{x_i=0}
\quad \text{for $i\in S_2$},
\]
where $x^{-v^{(i)}}=x_1^{-v_1}\cdots x_{i-1}^{-v_{i-1}}x_{i+1}^{-v_{i+1}}\cdots x_{n}^{-v_n}$.
Hence~\eqref{e-ind1} reduces to
\begin{multline*}
D_{v,\lambda}(a;n,n_0)
=\sum_{i\in S_1}\sum_{j=0}^{a_i-2}\CT_{x^{(i)},w^{(1)}}
q^{jr}x^{-v^{(i)}}w^{\lambda^{(1)}}(x_i^{-n+n_0}A_{ij})|_{x_i=0}
\\+\sum_{i\in S_2}\sum_{j=0}^{a_i-1}\CT_{x^{(i)},w^{(1)}}
q^{jr}x^{-v^{(i)}}w^{\lambda^{(1)}}B_{ij}|_{x_i=0}.
\end{multline*}
Using the expressions for $A_{ij}$ and $B_{ij}$ in~\eqref{A} and~\eqref{B}, respectively, we obtain
\begin{multline*}
D_{v,\lambda}(a;n,n_0)=\sum_{i\in S_2}\sum_{j=0}^{a_i-1}\frac{q^{j(|a|-a_i-n_0+r)+\sum_{l=i+1}^{n}a_l}}{(q^{-j})_j(q)_{a_i-j-1}}
\CT_{x^{(i)},w^{(1)}}x^{-v^{(i)}}w^{\lambda^{(1)}}F_{n-1,n_0}(a^{(i)};x^{(i)},w^{(1)})\\
+\sum_{i\in S_1}\sum_{j=0}^{a_i-2}
\frac{q^{j(|a|-a_i-n_0+r+1)+\sum_{l=i+1}^{n_0}(a_l-1)}}{(-1)^{n-n_0}(q^{-j})_{j}(q)_{a_i-j-2}}
\CT_{x^{(i)},w^{(1)}}x^{-(v+\alpha)^{(i)}}w^{\lambda^{(1)}}F_{n-1,n_0-1}(a^{(i)};x^{(i)},w^{(1)}).
\end{multline*}
Using~\eqref{e-sum} with $(k,t,n)\mapsto(j,a_i-1,|a|+r-1-n_0)$ and $(k,t,n)\mapsto(j,a_i-2,|a|+r-1-n_0)$, respectively, we have
\begin{multline*}
D_{v,\lambda}(a;n,n_0)=\sum_{i\in S_2}q^{\sum_{l=i+1}^n a_l} \qbinom{|a|+r-1-n_0}{a_i-1}
\CT_{x^{(i)},w^{(1)}}x^{-v^{(i)}}w^{\lambda^{(1)}}F_{n-1,n_0}(a^{(i)};x^{(i)},w^{(1)})\\
+\sum_{i\in S_1} \bigg((-1)^{n-n_0} q^{\sum_{l=i+1}^{n_0} (a_l-1)}\qbinom{|a|+r-1-n_0}{a_{i}-2}\\\times
\CT_{x^{(i)},w^{(1)}}x^{-(v+\alpha)^{(i)}}w^{\lambda^{(1)}}F_{n-1,n_0-1}(a^{(i)};x^{(i)},w^{(1)})\bigg).
\end{multline*}
By \eqref{e-relation} again, the constant term in the first sum equals to $D_{v^{(i)},\lambda^{(1)}}\big(a^{(i)};n-1;n_0\big)$, while that in the second sum equals to $D_{(v+\alpha)^{(i)},\lambda^{(1)}}\big(a^{(i)};n-1,n_0-1\big)$. This completes the proof.
\end{proof}

\section{Proof of Theorem~\ref{p-recu}}\label{sec-proof}

In this section, we give a proof of Theorem~\ref{p-recu} by the inductive formula~\eqref{e-ind}.

We first need the following result, which is essentially given by Zhou \cite[Proposition 5.4]{Zhou-JCTA}.
\begin{prop}\label{prop-eq}
Let $I$ be a non-empty subset of $\{1,2,\dots,m\}$ of cardinality at least two. Then for an integer $r$,
\begin{multline}\label{e-transsum}
\sum_{i\in I}\sum_{\emptyset \neq J\subseteq I\setminus \{i\}} (-1)^{|J|+1}q^{\sum_{j=i+1}^ma_j+L_{m,I\setminus\{i\},J}(a^{(i)})}
\frac{(1-q^{a_i})(1-q^{a_{J}})}{(1-q^{|a|-a_i+r})(1-q^{|a|-a_i-a_{J}+r})}\\
=\sum_{\emptyset \neq J\subseteq I}(-1)^{|J|}q^{L_{m,I,J}(a)}
\frac{1-q^{a_{J}}}{1-q^{|a|-a_{J}+r}},
\end{multline}
where $a_S:=\sum_{j\in S}a_j$ and $L_{m,I,J}(a)$ is defined as in \eqref{def-L}.
\end{prop}

\begin{proof}[Proof of Theorem~\ref{p-recu}]
   We prove the first case of the theorem by induction on $s_1$. Note that in this case $S_2=\emptyset$. For $s_1=1$, let $k$ be the unique element of $S_1$. By \eqref{e-ind} we have 
    \begin{equation}\label{e-recu-A1}
        D_{v,\lambda}(a;n,n_0)=(-1)^{n-n_0} q^{\sum_{l=k+1}^{n_0} (a_l-
        1)}\qbinom{|a|+r-1-n_0}{a_{k}-2}
        D_{(v+\alpha)^{(k)},\lambda^{(1)}}\big(a^{(k)};n-1,n_0-1\big).
    \end{equation}
    It is not hard to verify that the right-hand side of \eqref{e-recu-A1} equals the $s_1=1$ case of \eqref{e-recu-A}. 
    Suppose $2\leq s_1\leq n$. By \eqref{e-ind} and the induction 
    hypothesis, we find 
    \begin{equation}\label{e-recu-A2}\begin{aligned}
        D_{v,\lambda}(a;n,n_0)&=(-1)^{s_1(n-n_0)}
        \sum_{i\in S_1}q^{\sum_{l=i+1}^{n_0}(a_l-1)}
        \qbinom{|a|+r-1-n_0}{a_{i}-2}\frac{\qbinom{|\widetilde{a}|-\widetilde{a}_i+r-1}{\widetilde{a}^{(i)},r-1}}{\qbinom{|\widetilde{a}|-\widetilde{a}_{S_1}+r-1}{\widetilde{a}^{(S_1)},r-1}}
        \\&\quad\times
        \sum_{\emptyset \neq J\subseteq S_1\setminus \{i\}}(-1)^{|S_1\setminus  J|-1}q^{L_{n_0,S_1\setminus\{i\},J}(\widetilde{a}^{(i)})}
        \frac{1-q^{\widetilde{a}_{J}}}{1-q^{|\widetilde{a}|-\widetilde{a}_i-\widetilde{a}_{J}+r}}
        \\&\quad\times D_{(v+s_1\alpha)^{(S_1)},\lambda^{([s_1])}}\big(a^{(S_1)};n-s_1,n_0-s_1\big)
        .
    \end{aligned}\end{equation}
    It is easy to check that for $1\leq i \leq n$
    \begin{equation}\label{e-cal-q-bino}
        \qbinom{|a|+r-1-n_0}{{a}_i-2}\qbinom{|\widetilde{a}|-\widetilde{a}_i+r-1}{\widetilde{a}^{(i)},r-1}=\frac{1-q^{\widetilde{a}_i}}{1-q^{|\widetilde{a}|-\widetilde{a}_i+r}}\qbinom{|\widetilde{a}|+r-1}{\widetilde{a},r-1}.
    \end{equation}
     Therefore, we can rewrite \eqref{e-recu-A2} as 
    \begin{multline*}
        D_{v,\lambda}(a;n,n_0)=\sum_{i\in S_1}\sum_{\emptyset \neq J\subseteq S_1\setminus\{i\}}(-1)^{|J|+s_1(n-n_0+1)+1}q^{\sum_{l=i+1}^{n_0}\widetilde{a_l}+L_{n_0,S_1\setminus\{i\},J}(\widetilde{a}^{(i)})}
        \frac{\qbinom{|\widetilde{a}|+r-1}{\widetilde{a},r-1}}{\qbinom{|\widetilde{a}|-\widetilde{a}_{S_1}+r-1}{\widetilde{a}^{(S_1)},r-1}}
        \\\times\frac{(1-q^{\widetilde{a}_{J}})(1-q^{\widetilde{a}_i})}{(1-q^{|\widetilde{a}|-\widetilde{a}_i-\widetilde{a}_{J}+r})(1-q^{|\widetilde{a}|-\widetilde{a}_i+r})}
        D_{(v+|S_1\alpha|)^{(S_1)},\lambda^{([s_1])}}\big(a^{(S_1)};n-s_1,n_0-s_1\big)
        .
    \end{multline*}
    Applying~\eqref{e-transsum} with $(I,m,a)\mapsto (S_1,n_0,\widetilde{a})$ to the above identity, we obtain~\eqref{e-recu-A}.

    The proof of the second case is similar to that of the first case. Note that in this case, $S_1=\emptyset$.  We proceed by induction on $s_2$. For $s_2=1$, let $k$ be the unique element in $S_2$. By \eqref{e-ind}, we have
    \begin{equation}\label{e-recu-B1}
    D_{v,\lambda}(a;n,n_0)= q^{\sum_{l=k+1}^{n} a_l}\qbinom{|a|+r-1-n_0}{a_{k}-1}D_{v^{(k)},\lambda^{(1)}}\big(a^{(k)};n-1,n_0\big).
    \end{equation}
    It is not hard to verify that the right-hand side of \eqref{e-recu-B1} equals the $s_2=1$ case of \eqref{e-recu-B}.
    Suppose $2 \leq s_2\leq n$. By \eqref{e-ind} and the induction hypothesis, we find 
     \begin{multline}\label{e-recu-B2}
         D_{v,\lambda}(a;n,n_0)=\sum_{i\in S_2}q^{\sum_{l=i+1}^na_l}\qbinom{|a|+r-1-n_0}{a_i-1}\frac{\qbinom{|\widetilde{a}|-\widetilde{a}_i+r-1}{\widetilde{a}^{(i)},r-1}}{\qbinom{|\widetilde{a}|-\widetilde{a}_{S_2}+r-1}{\widetilde{a}^{(S_2)},r-1}}
         \\\times\sum_{\emptyset\neq J \subseteq S_2\setminus\{i\}}(-1)^{|S_2\setminus J|-1}q^{L_{n,S_2\setminus \{i\},J}(\widetilde{a}^{(i)})}\frac{1-q^{\widetilde{a}_{J}}}{1-q^{|\widetilde{a}|-\widetilde{a}_i-\widetilde{a}_{J}+r}} D_{v^{(S_2)},\lambda^{([s_2])}}(a^{(S_2);n-s_2,n_0}).
     \end{multline}
     Using \eqref{e-cal-q-bino} again, we can rewrite \eqref{e-recu-B2} as
\begin{multline*}
    D_{v,\lambda}(a;n,n_0)=\sum_{i\in S_2}\sum_{\emptyset\neq J \subseteq S_2\setminus\{i\}}(-1)^{|J|+1}q^{\sum_{l=i+1}^n\widetilde{a_l}+L_{n,S_2\setminus \{i\},J}(\widetilde{a}^{(i)})}
    (-1)^{s_2}\frac{\qbinom{|\widetilde{a}|+r-1}{\widetilde{a},r-1}}
{\qbinom{|\widetilde{a}|-\widetilde{a}_{S_2}+r-1}{\widetilde{a}^{(S_2)},r-1}}
    \\\times\frac{(1-q^{\widetilde{a_i}})(1-q^{\widetilde{a}_J})}{(1-q^{|\widetilde{a}|-\widetilde{a}_i+r})(1-q^{|\widetilde{a}|-\widetilde{a}_i-\widetilde{a}_J+r})}D_{v^{(S_2)},\lambda^{([s_2])}}(a^{(S_2);n-s_2,n_0}).
\end{multline*}
Applying~\eqref{e-transsum} with $(I,m,a)\mapsto (S_2,n_0,\widetilde{a})$ to the above identity, we obtain~\eqref{e-recu-A}.
\end{proof}

\section{Other forms of Theorem~\ref{Thm-van} and  Corollary~\ref{Thm-rec}}\label{Schur}
In this section, we show that in Theorem~\ref{Thm-van} and  Corollary~\ref{Thm-rec}, the complete symmetric functions $h_{\lambda}$ in $D_{v,\lambda}(a;n,n_0)$ can be replaced by the Schur functions $s_{\lambda}$. Denote 
\begin{equation*}
    D_{v,\lambda}^{\prime}(a;n,n_0)=\CT_{x}
x^{-v}s_{\lambda}\big(X^{(a;n_0)}\big)
\prod_{1\leq i<j\leq n}
(x_i/x_j)_{a_i-\chi(i\leq n_0)}(qx_j/x_i)_{a_j-\chi(i\leq n_0)},
\end{equation*}
where
\[
s_{\lambda}(x):=\frac{\det\big(x_i^{\lambda_i+n-j}\big)_{1\leq i, j \leq n}}{\prod_{1\leq i < j \leq n}(x_i-x_j)}
\]
is the Schur function for the partition $\lambda$ with length at most $n$.
We find that $D_{v,\lambda}^{\prime}(a;n,n_0)$ satisfies the same vanishing condition as $D_{v,\lambda}(a;n,n_0)$ in Theorem~\ref{Thm-van}.
\begin{prop}\label{Thm-van-Schur}
    Let $v \in \mathbb{Z}^{n}$ and $\lambda$ be a partition such that $|v|=|\lambda|$. Then $D_{v,\lambda}^{\prime}(a;n,n_0)=0$, if there exists an integer $0 < j < n$, such that 
    \begin{equation}\label{e-thm-van-Schur}
        \sum_{k=1}^j v_{i_k}-p_I(n-n_0-j+p_I)<\sum_{k=1}^j \lambda_k \quad \text{ for every 
        $I=\{i_1,\ldots,i_j\} \subseteq \{1,\ldots,n\}$},
    \end{equation}
    where $p_I:=\Mid \{ i_k : i_k\leq n_0, 1\leq k\leq j\} \Mid$ for a given $I$.
\end{prop}
\begin{proof}
    Using the Jacobi--Trudi identity \cite[Equation 3.4, Chapter 1]{Mac95}
    \begin{equation}\label{e-Jacobi}
        s_{\lambda}=\det(h_{\lambda-i+j})_{1\leq i, j \leq n}=h_{\lambda}+\sum_{\mu > \lambda}c_{\mu}h_{\mu},
    \end{equation}
    where $n \geq \ell(\lambda)$, we have 
    \[
    D_{v,\lambda}^{\prime}(a;n,n_0)=D_{v,\lambda}(a;n,n_0) + \sum_{\mu > \lambda}c_\mu D_{v,\mu}(a;n,n_0).
    \]
    By Theorem~\ref{Thm-van}, every summand in the right hand side of the equation above vanishes. Hence $D_{v,\lambda}^{\prime}(a;n,n_0)=0$.
\end{proof}
We can also obtain the following equivalence between $D_{v,v^+}^{\prime}(a;n,n_0)$ and $D_{v,v^+}(a;n,n_0)$. 
\begin{prop}
    \begin{equation}\label{e-equavelent}
    D_{v,v^+}^{\prime}(a;n,n_0)=D_{v,v^+}(a;n,n_0).
\end{equation}
\end{prop}
\begin{proof}
    Using the Jacobi--Trudi identity \eqref{e-Jacobi} again, we have
    \[
    D_{v,v^+}^{\prime}(a;n,n_0)=D_{v,v^+}(a;n,n_0) + \sum_{\mu > v^+}c_\mu D_{v,\mu}(a;n,n_0).
    \]
    By Corollary~\ref{Cor-van}, we have
    \[
    D_{v,\mu}(a;n,n_0)=0, \quad \text{for $\mu > v^+$}.
    \]
    Then we obtain \eqref{e-equavelent}.
\end{proof}

\section*{Acknowledgements}
This work was supported by the Science and Technology Innovation Program and Department of Education Funded Research Projects of Hunan
Province (No. 24A0025) and the National Natural Science Foundation of China (No. 12571359).

\end{document}